\documentclass{amsart}

\usepackage{amsmath,amsbsy,amsfonts,amsopn,amstext}
\usepackage{euscript,amssymb,amsbsy,amsfonts,amsthm,latexsym,amsopn,amstext,amsxtra,euscript,amscd}
\usepackage{graphicx, verbatim}

\usepackage{longtable}

\usepackage{xy} 
\input xy \xyoption{all}

\usepackage{graphicx}

\usepackage{hyperref}

\hypersetup{
  colorlinks   = true, 
  urlcolor     = blue, 
  linkcolor    = blue, 
  citecolor   = red 
}

\newtheorem{prop}{Proposition}

\newtheorem{defn}{Definition} 

\newtheorem{exa}{Example} 

\newtheorem{cor}{Corollary}

\DeclareMathOperator\Sl{SL}

\newcommand\Z{\mathbb Z}
\newcommand\Q{\mathbb Q}
\newcommand\R{\mathbb R}
\newcommand\C{\mathbb C}

\def\P{\mathbb P}

\def\H{\mathbb H}

\newcommand\F{\mathcal F}            
\newcommand\J{\mathcal J}

\def\<{\langle}

\newcommand\z{\xi}  

\usepackage{amsaddr}

\usepackage{tikz}
\usepackage{tikz,ifthen}
\usepackage{calc,rotating}
\usetikzlibrary{calc,fadings, intersections, decorations.pathreplacing}
\usetikzlibrary{arrows,positioning}
\usetikzlibrary{fit}
\usetikzlibrary{matrix}
\usepackage{tikz-3dplot}
\usepackage{fancyhdr}
\title{Equilibrium Laws for Julia's Zero and the Hyperbolic Zero of Binary Forms}

\author{A. Elezi}
\address{Department of Mathematics and Statistics \\ American University \\ Washington, DC, 20016. \\
Email: \, aelezi@american.edu}



\date{}                                           

\begin{document}
\pagestyle{fancy}
\fancyhead[L]{Artur Elezi}
\fancyhead[R]{Julia's Zero and the Hyperbolic Zero}
\theoremstyle{remark}
\newtheorem{remark}{Remark} 


\begin{abstract}
For a real binary form with no real roots, Julia's zero $\z_{\J}(F)$ and the
hyperbolic zero $\z_H(F)$ are two $\Sl_2(\R)$-equivariant
points in the upper half-plane $\H^2$. We show that they admit the
closely related equilibrium characterizations
\[
\sum_i (\tanh t_i){\bf v}_i=0
\qquad\text{and}\qquad
\sum_i (\sinh t_i){\bf v}_i=0,
\]
respectively, where $t_i$ and ${\bf v}_i$ are the hyperbolic distance and unit
tangent direction from the candidate point to the $i$th root, with zero weighted contribution when the two points coincide. This common
framework gives geometric criteria for coincidence of the two zero maps.
They always agree for binary quartics; for binary sextics they agree
exactly when the three upper-half-plane roots form an equilateral
hyperbolic triangle, or are collinear with one root the hyperbolic midpoint
of the other two. We also obtain a corresponding result for collinear
binary octics. The two equilibrium laws further reveal a sharp difference
in the influence of distant roots. A strict majority of roots confined to
a compact set keeps Julia's zero in a compact set, and the threshold
one-half is optimal. In contrast, an escaping minority can force the
hyperbolic zero to escape at linear scale in the Euclidean upper-half-plane coordinates. We also illustrate the
computational advantages of the explicit formula for the hyperbolic zero.
\end{abstract}
\maketitle

\

\section{Introduction}

\

\noindent  During the nineteenth century, considerable effort was devoted to developing
a reduction theory for binary forms, in the spirit of the classical reduction
theory of quadratic forms. One of the underlying ideas can be expressed
geometrically as follows. Given a set $A$ carrying a right
$\Sl_2(\Z)$-action, associate to each $a\in A$ a covariant point in
the upper half-plane $\H^2$; that is, construct an $\Sl_2(\Z)$-equivariant
map
\[
\xi:A\longrightarrow\H^2.
\]

\noindent  Since the modular group acts on binary forms by substitution and on $\H^2$
by M\"obius transformations, such a map gives a natural way of selecting
representatives in an $\Sl_2(\mathbb Z)$-orbit. The arithmetic objective is
to find representatives with small coefficients, or small height. For
quadratic forms this is classical; in higher degree the problem is
considerably more subtle.

\

\noindent  Julia gave the first substantial general answer in his 1917 thesis
\cite{Julia}, where he developed a reduction theory for real binary forms.
To a real binary form $F$, Julia associated a positive definite quadratic
$J_F$, now called the Julia quadratic. Positive definite real quadratic
forms, up to positive scaling, are parametrized by $\H^2$, so $J_F$
determines a point
\[
\z_{\J}(F)\in\H^2,
\]
the Julia zero of $F$. A form is Julia-reduced when this point lies in the
standard fundamental domain for $\Sl_2(\mathbb Z)$.

\

\noindent Cremona \cite{Cremona} revisited Julia's construction, giving explicit
treatments in degrees three and four, and Stoll and Cremona
\cite{CS} subsequently developed a unified theory for real and
complex binary forms. In the complex setting the corresponding covariant
is a positive definite Hermitian form, and positive definite binary
Hermitian forms up to scaling are parametrized by hyperbolic three-space
$\H^3$. The real theory appears as the totally real slice
$\H^2\subset\H^3$. Thus hyperbolic geometry provides a natural common
framework for these reduction theories.

\

\noindent  The subject has seen renewed activity. Rosu \cite{Rosu} studies covariant
points of binary forms and their relation to the geometry and distribution
of the roots. Stoll \cite{Stoll} develops efficient methods for computing
and reducing binary forms using their associated covariant points. Recent
computational work of Kotsireas and Shaska \cite{KS}
investigates geometric reduction on a large scale and illustrates that
different geometric choices can lead to different reduced representatives.
Related work on heights by Shaska and Beshaj \cite{BS} and Beshaj
\cite{B} makes the relation between geometric reduction and the
arithmetic size of the resulting coefficients particularly relevant.

\

\noindent The purpose of this paper is to compare Julia's zero with another
geometrically natural zero map, the hyperbolic zero, and to understand both
when the two constructions agree and when they behave differently. The
hyperbolic center of mass goes back to Galperin \cite{Gal}, who
introduced it in the hyperboloid model of hyperbolic space. In
\cite{ES} we carried this construction over to the upper
half-plane model. For points
\[
\alpha_1,\ldots,\alpha_n\in\H^2,
\]
with all roots counted with multiplicity throughout this paper, their hyperbolic center of mass is the unique point minimizing
\[
\sum_{i=1}^n\cosh d_H(P,\alpha_i).
\]

\noindent Applied to the upper-half-plane roots of a totally complex real binary
form, this defines an $\Sl_2(\mathbb R)$-equivariant zero map, which we
denote by $\z_H$. We also obtained in \cite{ES} an explicit
formula for $\z_H$ in every degree when the form is given as a product
of real quadratic factors.

\

\noindent  The starting point of the present paper is that Julia's zero and the
hyperbolic zero can be put into the same geometric framework. Let
\[
t_i=d_H(P,\alpha_i)
\]
\noindent  and, when $t_i>0$, let ${\bf v}_i$ be the unit tangent vector at $P$ pointing toward
$\alpha_i$. When $t_i=0$, both weighted vectors $(\tanh t_i){\bf v}_i$ and
$(\sinh t_i){\bf v}_i$ are defined to be zero, their continuous limits.
We use this convention throughout. We show that
\[
P=\z_{\J}(F)
\quad\Longleftrightarrow\quad
\sum_{i=1}^n(\tanh t_i){\bf v}_i=0,
\]
whereas the hyperbolic zero satisfies
\[
P=\z_H(F)
\quad\Longleftrightarrow\quad
\sum_{i=1}^n(\sinh t_i){\bf v}_i=0.
\]

\noindent  Thus the two zero maps are described by equilibrium laws of exactly the
same form, with the difference contained entirely in the radial weights
$\tanh t$ and $\sinh t$. This gives a direct way of comparing two
constructions which arise from rather different definitions.

\

\noindent  This comparison already gives useful information about when the two maps
coincide. We obtain a general symmetry principle and then consider the
problem in low degree. For binary quartics the two zero maps always agree.
For binary sextics we characterize coincidence completely: the three
upper-half-plane roots either form an equilateral hyperbolic triangle, or
they are collinear with one root the hyperbolic midpoint of the other two.
For collinear binary octics we obtain a corresponding characterization in
terms of pairs having the same hyperbolic midpoint. These results are
obtained directly from the two equilibrium equations and do not require
explicit computation of either zero map.

\

\noindent  The difference between the two weights becomes more significant when some
of the roots move far away:
\[
\tanh t\longrightarrow1,
\qquad
\sinh t\longrightarrow\infty
\qquad (t\longrightarrow\infty).
\]

\noindent  This leads to a sharp difference between the influence of distant roots
on the two constructions. We prove that if a strict majority of the
upper-half-plane roots remains in a fixed compact subset of $\H^2$, then
Julia's zero remains in a fixed compact set, independently of the positions
of all the remaining roots. Moreover, the strict-majority hypothesis is
sharp: if only half of the roots remain controlled, Julia's zero need not
remain bounded. Thus the threshold $1/2$ is intrinsic to the bounded
weight $\tanh t$.

\

\noindent  The behavior of the hyperbolic zero is quite different. Even when a strict
majority of the roots remains fixed, a minority of roots moving
horizontally to infinity can force $\z_H(F)$ to escape at linear scale in the Euclidean upper-half-plane coordinates.
We give an explicit family with three fixed upper-half-plane roots and two
escaping roots for which
\[
\z_{\J}(F_M)
\longrightarrow
1+i\sqrt{\frac{19+\sqrt{761}}{2}},
\qquad
|\z_H(F_M)|\asymp M.
\]

\noindent  Thus the boundedness of $\tanh t$ and the unbounded growth of $\sinh t$
produce genuinely different responses to distant roots. This distinction is not special to a single outlier; it persists for an escaping minority of roots.

\

\noindent  There is also a computational difference between the two constructions.
For totally complex real forms given as products of real quadratic factors,
the hyperbolic zero has an explicit formula in terms of the coefficients
of those factors. Julia's zero is in general determined by a nonlinear
minimization or equilibrium problem. We give an example in which the two
zero maps lead to different reduced representatives and the representative
obtained from hyperbolic reduction has substantially smaller coefficient
height. This example is not intended to suggest that one reduction always
produces smaller coefficients than the other, but it shows that the
difference between the two zero maps can have arithmetic as well as
geometric consequences.

\

\noindent  The paper is organized as follows. We first recall the hyperbolic geometry
of $\H^2$ and $\H^3$ needed for the reduction theory and review Julia's
construction and the Cremona--Stoll interpretation. We then derive the
equilibrium characterization of Julia's zero for totally complex real
forms and recall the hyperbolic center of mass and its associated zero map.
We next compare the two equilibrium laws and study coincidence in low
degree. In Section~\ref{sec:explicit-asymptotic} we recall the explicit
formula for the hyperbolic zero, compare the two reductions
computationally, and study their behavior when a minority of the roots
escapes. We conclude with
some questions suggested by the comparison.

\
\noindent Throughout, the symbol $\dagger$ marks the end of a statement,
definition or remark that is not followed by a proof; proofs end with the
usual $\square$.

\

\section{The hyperbolic plane as a parameter space for positive definite quadratic forms}\label{sec:hypplane} 

\

\noindent In this section we recall the features of the hyperbolic plane that will
be needed later and review its correspondence with positive definite
quadratic forms.

\subsection{The Poincar\'e model of the hyperbolic plane \texorpdfstring{$\H^2$}{H2}} The upper-half plane is one of the models of the two-dimensional hyperbolic space. It is denoted by $\H^2$.  The geodesics of the Riemannian manifold $\H^2$, i.e, the hyperbolic equivalents of Euclidean straight lines, are either semicircles $C_{a,b}$ with diameter from $A(a,0)$ to $B(b,0)$ on the real axis, or the vertical rays $C_a$ with origin at $x=a$. In the standard literature, the points $A(a,0), B(b,0)$ are called {\bf the ideal points} of the geodesic $C_{a,b}$, likewise $A(a,0)$ and $\infty$ are the ideal points of $C_a$. The ideal points of the geodesic live in the boundary of $\H^2$.






\noindent Let $z=x+{\bf i}y,~w=u+{\bf i}v\in \H^2$. Let $z_{\infty}, w_{\infty}$ be the ideal points of the geodesic through $z,w$, where $z_{\infty}$ is the one closer to $z$. The hyperbolic distance between $z,w$ is defined as follows

\[ d_H(z,w)=\cosh^{-1}\left(1+\frac{|z-w|^2}{2yv}\right)\]
Notice that for $x=u$ and $y<v$, the geodesic is the vertical ray $C_x$. In this case \break $z_{\infty}=(x,0), w_{\infty}=\infty$ and 
\[ d_H(z,w)=\ln \left(\frac{v}{y}\right).\] 

\

\begin{figure}[htbp] 
   \centering

\begin{tikzpicture}[help lines/.style={blue!30,very thin},scale=0.6]
   \draw [help lines] (-2, 0) grid (16, 5);
    \draw[->, thick] (-2, 0) -- (16, 0) node[below]{\footnotesize $x$};
    \draw[->, thick] (0, 0) -- (0, 5) node[right]{\footnotesize $y$};

\draw[red, very thick ] (7, 0) arc (0:180:3cm);

    \draw[-, color=red,very thick] (11, 0) -- (11, 5);

    \node at (11.5, 2) {$z$};
  \node at (11.5, 4) {$w$};
    
    \foreach \x/\y in {   4/3, 2/2.2, 11/2, 11/4}
    \filldraw[blue] (\x, \y) circle(3pt);

    \node at (2, 2.7) {$z$};
    \node at (4, 3.5) {$w$};
     
    \node at (1, -0.5) {$z_{\infty}$};
    \node at (7, -0.5) {$w_{\infty}$};

\end{tikzpicture}

   \caption{The hyperbolic distance between two points $z$ and $w$ with $\Re (z) \neq \Re(w) $ and $\Re (z) = \Re(w) $ }
   \label{fig1}
\end{figure}
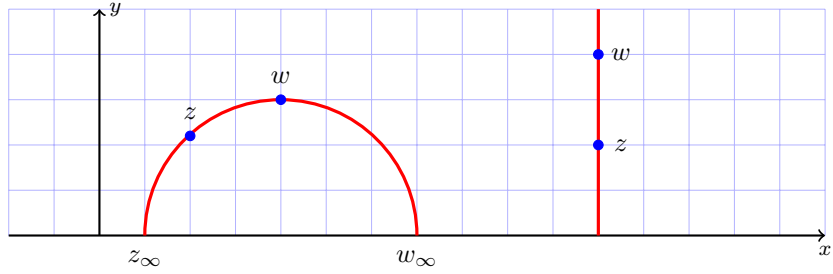






%

\noindent For $A(a,0)$ and $z=x+{\bf i}y\in \H^2$, an extended boundary distance function is defined following Cremona-Stoll \[ d_H(A,z):=\ln\left(\frac{(x-a)^2+y^2}{y}\right).\] For the ideal endpoint at infinity, put $d_H(\infty,z)=-\log y$. We use the same value when the arguments of an extended boundary distance are reversed.

\noindent The following additive property of this distance is claimed and used in \cite{CS} and proven in \cite{ES}. 

\begin{prop} 
Let $A$ be the ideal endpoint of a geodesic that passes through $z,w\in \H^2$ and is closer to $z$. Then $$d_H(A,z)+d_H(z,w)=d_H(A,w).\dagger$$ 
\end{prop}

\noindent The group $\Sl_2(\R)$ acts on the right on $\H^2$: if $M\in \Sl_2(\R)$ and $M^{-1}=\begin{pmatrix}  a & b \\ c & d  \end{pmatrix}$ then $$\displaystyle{z\cdot M:=M^{-1}z=\frac{az+b}{cz+d}}$$

\subsection{The hyperboloid model} Consider the Minkowski pairing in $\R^{2,1}$: 
\[ M({\bf x}, {\bf y})=-x_1y_1-x_2y_2+x_3y_3\]
 and the corresponding norm $||{\bf x}||^2=M({\bf x},{\bf x})$. Let $\H$ be the upper sheet of the hyperboloid \[\H:=\{{\bf x}=(x,y,z)\in \R^3~:~||{\bf x}||=1,~z>0\}.\] Its equation is $-x^2-y^2+z^2=1$. An isometry $\H^2\rightarrow \H$ is given by 
 \[ u+{\bf i}v \rightarrow \left(   \frac{1-u^2-v^2}{2v}, \frac{u}{v},\frac{1+u^2+v^2}{2v} \right).\]
  If ${\bf x},{\bf y}\in \H$, the hyperbolic distance $d_H({\bf x},{\bf y})$ in this model can be found via $$\cosh d_H({\bf x},{\bf y})=M({\bf x},{\bf y}).$$

\begin{defn}\label{def:centerofmass}\cite{Gal}
Let ${\bf x}_j\in \H,~j=1,2,...,r$. Their {\bf center of mass} is defined as $${\mathcal C}={\mathcal C}_H({\bf x}_1,{\bf x}_2,...,{\bf x}_r):=\displaystyle{\frac{\sum_{j=1}^r{\bf x}_j}{||\sum_{j=1}^r{\bf x}_j||}}.\dagger$$ 
\end{defn}

\

\noindent The following two propositions were established in \cite{ES}.

\begin{prop}
The center of mass ${\mathcal C}_H({\bf x}_1,{\bf x}_2,...,{\bf x}_r)$ is $\Sl_2(\R)$-equivariant. It is the unique point ${\bf x}\in \H$ that minimizes $\sum_{j=1}^r \cosh (d_H({\bf x},{\bf x}_j))\dagger$. 
\end{prop}

\begin{prop} Write $x_j+{\bf i}y_j\in\H^2$ for the point corresponding to ${\bf x}_j\in\H$ under the model isometry, for $j=1,\ldots,r$. The hyperbolic center of mass ${\mathcal C}_H({\bf x}_1,{\bf x}_2,...,{\bf x}_r)$ is the unique point $u+{\bf i}v\in \H^2$ that minimizes $$\sum_{j=1}^r \frac{(u-x_j)^2+(v-y_j)^2}{vy_j}.\dagger$$ 
\end{prop}

\

\noindent In this paper we provide yet another characterization of the hyperbolic center of mass that will be useful. First, we describe the geodesics in the hyperboloid model.

\

\noindent  For simplicity, let $P=(0,0,1)\in \H$, the general case follows by $\Sl_2 (\R)$ invariance. A geodesic through $P$ is uniquely determined by a vector that must be perpendicular to $\overrightarrow{OP}$ and of unit length:  ${\bf v}=\left<v_1,v_2,0\right>,~v_1^2+v_2^2=1$. The plane that contains the origin, $P$ and $\bf v$ intersects $\H$ in a geodesic. The parametric equation of this geodesic is 

\begin{equation} {\bf x}(t)=(\cosh t) P+(\sinh t){\bf v} \label{geodesic}
\end{equation}

\

\noindent where $t$ is the hyperbolic distance from $P$ to ${\bf x}(t)$.

\

\begin{prop} \label{centerfromgeometry} Let ${\bf x}_1,{\bf x}_2,...,{\bf x}_r$ be points in the hyperboloid $\H$. Let $t_i$ be the hyperbolic distance from $P$ to ${\bf x}_i$ and ${\bf v}_i$ the unit tangent vector at $P$ in the geodesic direction from $P$ to ${\bf x}_i$. The center of mass ${\mathcal C}_H({\bf x}_1,{\bf x}_2,...,{\bf x}_r)$ is the unique point that satisfies $$\displaystyle{\sum_{i=1}^r (\sinh t_i){\bf v}_i=0},$$ where $t_i=d_H(P,{\bf x}_i)$ and ${\bf v}_i$ is the unit tangent vector from $P$ toward ${\bf x}_i$ when $t_i>0$, with zero weighted contribution when $t_i=0$.
\end{prop}

\begin{proof} We compute \[\sum_{i=1}^r {\bf x}_i=(\sum_{i=1}^r \cosh t_i)P+\sum_{i=1}^r (\sinh t_i){\bf v}_i\]

\noindent The point $P$ is the hyperbolic center of ${\bf x}_1,...,{\bf x}_r$ iff \[\sum_{i=1}^r{\bf x}_i=||\sum_{i=1}^r{\bf x}_i||P,\]  i.e. iff $\displaystyle{\sum_{i=1}^r (\sinh t_i){\bf v}_i=0}$ in which case $\displaystyle{||\sum_{i=1}^r{\bf x}_i||}=\sum_{i=1}^r \cosh t_i$. 
\end{proof}

\

\begin{cor} (of the proof). If $P$ is the center of mass, then with the same notations as before $\displaystyle{||\sum_{i=1}^r{\bf x}_i||}=\sum_{i=1}^r \cosh t_i. \dagger$
\end{cor}

\

\subsection{\texorpdfstring{$\H^2$}{H2} as a parameter space for positive definite real quadratic forms} Let $$Q(X,Z)=aX^2-2bXZ+cZ^2$$ be a binary quadratic form with real coefficients and homogeneous variables $[X,Z]\in\P^1\R$. Let $\Delta=ac-b^2$ be its discriminant. Then $$Q(X,Z)=a[X-(b/a)Z]^2+(\Delta/a)Z^2.$$ For both $\Delta>0$ and  $a>0$, $Q(X,Z)$ is always positive (note that $(X,Z)\neq (0,0)$ since $[X,Z] \in \P^1\R$). Such a quadratic form $Q$ is called {\bf positive definite}.  It has two roots $[\omega_{Q},1],[\overline {\omega_{Q}},1]\in \P^1\bf C$ where $\omega_{Q}=b/a+(\sqrt {\Delta}/a) {\bf i}\in \H^2$. 

\ 

\begin{defn}\label{def:zeromap}Let $V^+_{2,\R}$ be the space of positive definite real quadratic forms. The map $$\z:V^+_{2,\R}\rightarrow \H^2$$ which sends a positive definite quadratic form $Q(X,Z)$  to its root $\omega_{Q}\in \H^2$ is called {\bf the zero map}. 
\end{defn}

\noindent The hyperbolic plane $\H^2$ parametrizes faithfully positive definite quadratic forms up to a constant factor via the inverse 
\[ \z^{-1}(\omega)=Q_{\omega}:=(X-\omega Z)(X-\bar{\omega}Z).\]
 The group $\Sl_2(\R)$ acts on the right on $V^+_{2,\R}$ via the linear change of variables: for a matrix $M= \begin{pmatrix}  a & b \\ c & d  \end{pmatrix}$, $$(Q\cdot M)(X,Z)=Q^M(X,Z):=Q(aX+bZ,cX+dZ).$$ Note that the $\Sl_2(\R)$ action does not change the discriminant. The zero map $\z: V^+_{2,\R}\rightarrow \H^2$ is $\Sl_2(\R)$-equivariant, i.e. $$\z(Q\cdot M)=M^{-1}\z(Q).$$

\noindent When $\Delta=0$, the quadratic form $Q(X,Z)=a[X-(b/a)Z]^2$ has a real double root $[b/a,1]$. If $a$ is a real number, we let $Q_a=(X-aZ)^2$ be the quadratic form that has a double root at $[a,1]$ while $Q_{\infty}=Z^2$ will be the quadratic form with a double root at $\infty$. We have thus established that the boundary $\R\P^1=\R \cup \infty$ of $\H^2$ parametrizes quadratic forms (up to a constant factor) with discriminant $\Delta=0$.

\

\noindent To recap: the hyperbolic plane $\H^2$ parametrizes binary quadratic forms with discriminant $\Delta>0$ and $a>0$, while its boundary parametrizes those with discriminant $\Delta=0$.

\

\noindent It has been claimed and used in \cite{Julia} and \cite{CS} that this parametrization is not just a bijection between sets; the hyperbolic geometry of $\H^2$ represents faithfully the algebra of quadratic forms. We believe that this was probably known even before. The following statements were formulated and proven in \cite{ES}.
\begin{prop}
Let $\overline{\H}^2=\H^2\cup \partial \H^2=\H^2\cup\R\P^1$ and $\omega_1,\omega_2\in \overline{\H}^2$. The quadratics of the form $$s Q_{\omega_1}+t Q_{\omega_2},s \geq 0,t\geq 0, s+t=1$$ parametrize the hyperbolic segment that joins $\omega_1$ and $\omega_2 .~\dagger$
\end{prop}

\noindent This is generalized by induction as follows. 

\begin{prop}\label{convex hull}
Let $\omega_1,\omega_2,...,\omega_n\in \overline{\H}^2$ such that for all $i$, $\omega_i$ is not in the hyperbolic convex hull of $\omega_1,\omega_2,...,\omega_{i-1}$. Then the convex hull of  $\omega_1,\omega_2, \dots ,\omega_n$ parametrizes the linear combinations $\sum_{i=1}^n\lambda_iQ_{\omega_i}$ with $\lambda_i\geq 0$ and $\sum_{i=1}^n\lambda_i=1.~ \dagger$
\end{prop}

\

\section{The Hyperbolic 3-space as a parameter space for positive definite quadratic Hermitian Forms}

\

\subsection{The hyperbolic three dimensional space \texorpdfstring{$\H^3$}{H3}} As a set, $\H^3=\C \times \R^+$. Points of $\H^3$ will be written in the form $z+t{\bf j}$ where $z\in \C$ and $t>0$.  The equation $t=0$ represents the floor $\C$ of $\H^3$. The hyperbolic space $\H^3$  is foliated via horospheres %
\[H_t:=\{z+t{\bf j}:~z\in \C\}\] 
which are centered at $\infty$ and indexed by the height $t$ above $\partial \H^3=\P^1\C$. There is a natural isometrical inclusion map $\H^2\rightarrow \H^3$ via $x+{\bf i}t\rightarrow x+{\bf j}t$. The invariant elements of $\H^3$ under the partial conjugation \[z+{\bf j}t \mapsto \bar z+{\bf j}t\] are precisely the elements of $\H^2$. The geodesics are either semicircles centered on $\C$ and perpendicular to $\C$, or rays $\{z_0+{\bf j}t\}$ perpendicular to $\C$. 

\

\noindent For $\omega=z+t{\bf j}\in \H^3$ and $w+0{\bf j}\in \C$ in the floor, an extended distance function is defined $$d_H(\omega,w):=\ln\frac{|z-w|^2+t^2}{t}.$$ At the remaining ideal boundary point, define $d_H(\omega,\infty)=-\log t$. We write $d_H$ also for the ordinary hyperbolic distance between two interior points; which of the two is meant is determined by whether the second argument lies on the ideal boundary. The following proposition and its proof are straightforward generalizations from $\H^2$.

\begin{prop}
Let $\omega_1,\omega_2\in\H^3$ and let $\omega$ be the ideal endpoint of the geodesic through them on the $\omega_1$ side, so that $\omega_1$ lies between $\omega$ and $\omega_2$. Then $$d_H(\omega_2,\omega)=d_H(\omega_1,\omega)+d_H(\omega_1,\omega_2),$$ the first two distances being extended distances to the ideal boundary point $\omega$ and the third the ordinary distance.~$\dagger$ 
\end{prop}

\noindent There is a right action of $\Sl_2(\C)$ on $\H^3$ which when restricted to the floor $t=0$ yields the standard $\Sl_2(\C)$-action on $\C\P^1$.

\subsection{\texorpdfstring{$\H^3$}{H3} and  positive definite Hermitian quadratic forms} Let $$H(X,Z)=a|X|^2-bX\bar Z-\bar b\bar XZ+c|Z|^2, a,c\in \R$$ be a Hermitian quadratic form with homogeneous variables $[X,Z] \in \P^1\C$. Notice that the values of $H(X,Z)$ are always real. Let $\Delta=ac-|b|^2$ be its discriminant. For $a\neq0$, $$H(X,Z)=a|X-(\bar b/a)Z|^2+(\Delta/a)|Z|^2.$$ If both  $\Delta >0,a>0$ then $H(X,Z)>0$ for all nonzero $(X,Z)$. Such a form is called {\bf positive definite.}  The set of all positive definite Hermitian quadratic forms is denoted here by $V^+_{2,\C}$. There is a right $\Sl_2(\C)$ substitution action on $V^+_{2,\C}$, similar to the real case. The natural $\Sl_2(\R)$ equivariant inclusion $\psi: V^+_{2,\R}\rightarrow V^+_{2,\C}$ via $$\psi(aX^2-2bXZ+cZ^2)=a|X|^2-bX\bar Z-\bar b\bar XZ+c|Z|^2,$$ gives rise to an extension of the zero map.

\begin{defn}
The zero map $\z: V^+_{2,\C}\rightarrow \H^3$ is defined via
\begin{equation} \label{eq: hermitian zero map}
\z(a|X|^2-bX\bar Z-\bar b\bar XZ+c|Z|^2)=\frac{\bar b}{a}+{\bf j}\frac{\sqrt \Delta}{a}.~\dagger
\end{equation}
\end{defn}

\noindent This map $\z$ is $\Sl_2(\C)$-equivariant.  It extends the zero map
of Definition~\ref{def:zeromap}; the symbol $\z_{\J}$ is reserved for Julia's zero map of a
binary form, introduced in Section~\ref{sec:julia}.

\
 
\noindent The hyperbolic space $\H^3$ parametrizes (up to a constant factor) positive definite ($\Delta>0,a>0$) Hermitian quadratic forms via the inverse map $$H_{\omega}=\z^{-1}(\omega)=\z^{-1}(z+{\bf j}t)= |X-zZ|^2+t^2|Z|^2,$$ while the boundary $\C\P^1=\C\cup{\infty}$ of $\H^3$ parametrizes the decomposable ($\Delta=0$) Hermitian forms \[H_{\beta}=(X-\beta Z)(\bar X-\bar\beta\bar Z)~\text{for}~\beta\in \C,~\text{and}~H_{\infty}=|Z|^2,\] Just as in the case of $\H^2$, we have the following proposition:

\begin{prop} \label{convex hull H3}
Let $\overline{\H}^3=\H^3\cup \partial \H^3=\H^3\cup \C\P^1$ and let $\omega_1,\omega_2,...,\omega_n\in \overline{\H}^3$ be such that for all $i$, $\omega_i$ is not in the hyperbolic convex hull of $\omega_1,\ldots,\omega_{i-1}$. Then the hyperbolic convex hull of $\omega_1,\omega_2,...,\omega_n$ parametrizes Hermitian forms $\sum_{i=1}^n \lambda_i H_{\omega_i}$ with $\lambda_i\geq 0$ for $i=1,2,...,n$ and $\sum_{i=1}^n \lambda_i=1.~\dagger$
\end{prop}

\noindent The equivariant connection between the geometry of hyperbolic spaces and the algebra of positive definite forms, which extends to the boundary as well, can be expressed in the following equivariant commutative diagram:

\[
\xymatrix{
 V^+_{2,\R}   \ar@{->}[d]  \ar@{->}[r]^\z      & \H^2  \ar@{->}[d]^{}  \\
 V^+_{2,\C}        \ar@{->}[r]^\z                   & \H^3 \\
}
\]

\

\section{Reduction of binary forms via Julia's Zero Map}\label{sec:julia} 

\

\noindent In this section we summarize the reduction of binary forms via Julia's zero map obtained in \cite{Julia} and \cite{CS}. We will focus on the geometric features of the theory which are of special interest to us.

\

\noindent Let $V_{n,\C}$ denote the space of complex binary forms of degree $n$. If $F\in V_{n,\C}$ has no root at infinity, then
 \[ F(X,Z)=a_0\prod_{i=1}^n(X-\alpha_iZ) \]
for some complex numbers $\alpha_j$ and $a_0\neq 0$. More generally, a nonzero form with a root of multiplicity $s$ at infinity has the factorization
\[
F(X,Z)=a_0Z^s\prod_{i=1}^{n-s}(X-\alpha_iZ),\qquad a_0\ne0.
\]
In all root sums below, roots at infinity are included with multiplicity and contribute the extended distance $d_H(w,\infty)=-\log t$ for $w=z+t{\bf j}$. The finite roots of $F$ lie in the floor $t=0$ of $\H^3$, and a root at infinity lies at the remaining ideal boundary point. We will employ the extended distance function between a point $w=z+{\bf j}t\in \H^3$ and $\alpha \in \C$ in the boundary floor: $$d_H(w,\alpha)=\ln \frac {|z-\alpha|^2+t^2}{t}.$$

\noindent {\bf Remark} Recall that a  binary form of degree $N$ is called {\bf stable} if no root has multiplicity at least $N/2$. In what follows, whenever we invoke the Stoll--Cremona Julia covariant and its uniqueness properties in $\H^3$, we assume that $F$ is stable. We denote the stable locus by $V^{\mathrm{st}}_{n,\C}\subset V_{n,\C}$.

\noindent For a totally complex real form the only failure of stability is
the degenerate one: if some root has multiplicity at least $N/2$, then its
conjugate has the same multiplicity, so both multiplicities equal $N/2$ and
all upper-half-plane roots coincide at a single $\alpha$. Such a form is a
scalar multiple of a power of a positive definite real quadratic, and for it
we use the real-slice extension described after Proposition $5.2$ of
\cite{CS}, namely $\z_{\J}(F)=\alpha$, the common root. This is also
$\z_H(F)$, since the hyperbolic center of a configuration concentrated at one
point is that point. With this convention both zero maps are defined for every
totally complex real form, and they agree on every nonstable one.

\begin{defn}
(Julia and Proposition $5.3$ in \cite{CS}) Let $F(X,Z)$ be a stable binary form. There is a unique point $\z_{\J}(F)\in  \H^3$ that minimizes the sum of distances $$\displaystyle{\sum_{i=1}^n d_H(w,\alpha_i)}$$
The map $\z_{\J}: V^{\mathrm{st}}_{n,\C}\rightarrow \H^3$ is called {\bf Julia's zero map}. The positive definite Hermitian quadratic form $H_{\z_{\J}(F)}\in V^+_{2,\C}$ is called the {\bf Julia quadratic} of $F.~\dagger$
\end{defn}

\noindent For a stable real binary form $F$, conjugation invariance and uniqueness place $\z_{\J}(F)$ in the real slice $\H^2$. Let
\[
\F=\{z\in\H^2: |\Re z|\leq\tfrac12,\ |z|\geq1\}
\]
be the standard closed fundamental region for $\Sl_2(\Z)$, with the usual boundary identifications. If $\z_{\J}(F)\in\F$, the form is called {\bf Julia-reduced}. Otherwise choose $M\in\Sl_2(\Z)$ with $\z_{\J}(F)\cdot M\in\F$; then $F^M$ is Julia-reduced. Here the right substitution action is
\[
F\cdot M=F^M,\qquad F^M(X,Z)=F(aX+bZ,cX+dZ),
\quad M=\begin{pmatrix}a&b\\c&d\end{pmatrix},
\]
so that $(F^M)^N=F^{MN}$ and $\z_{\J}(F^M)=\z_{\J}(F)\cdot M$. The same reduction convention applies to the nonstable totally complex real forms using the extension above. Geometric reduction does not in general minimize coefficient height.

\

\noindent Another equivalent, geometric description of the zero map is given by the following statement:

\begin{cor}\label{tangent vectors}
(Corollary $5.4$ in \cite{CS}) Let $F$ be a stable binary form. The zero map value $\z_{\J}(F)$ is the unique point where the unit tangent vectors along the geodesics to the roots $\alpha_i$ add up to zero.~$\dagger$
\end{cor}

\

\section{Julia's zero map for totally complex real binary forms}

\

\noindent  We now specialize Julia's construction to real binary forms with no real
roots. In this case the roots occur in conjugate pairs, and Julia's zero
lies in the natural copy of $\H^2$ inside $\H^3$. This leads to an
equilibrium condition entirely in $\H^2$, which will be useful for comparing
Julia's zero with the hyperbolic zero.

\noindent  Let
\[
F(X,Z)
=
a_0\prod_{i=1}^n
(X-\alpha_iZ)(X-\overline{\alpha_i}Z),
\qquad
\alpha_i\in\H^2,~a_0\in \R.
\]
\noindent  Since $F$ has real coefficients, $\z_{\J}(F)$ is fixed by conjugation and
therefore
\[
\z_{\J}(F)\in\H^2\subset\H^3.
\]
For $P\in\H^2$, let
\[
t_i=d_H(P,\alpha_i),
\]
and, for $t_i>0$, let ${\bf v}_i\in T_P\H^2$ be the unit tangent vector at $P$ pointing toward
$\alpha_i$. Weighted contributions at $t_i=0$ are zero, as in the introduction.

\begin{prop}
\noindent  Let $F(X,Z)$ be a real binary form with no real roots, and let
$\alpha_1,\ldots,\alpha_n\in\H^2$ be its upper-half-plane roots. Then
\[
P=\z_{\J}(F)
\]
\noindent  if and only if
\[
\sum_{i=1}^n \tanh(t_i)\,{\bf v}_i=0,
\]
\noindent  where $t_i=d_H(P,\alpha_i)$ and ${\bf v}_i$ is the unit tangent vector at $P$
pointing toward $\alpha_i$ when $t_i>0$, with zero weighted contribution when $t_i=0$.
\end{prop}

\begin{proof}
\noindent  By $\Sl_2(\R)$-equivariance, we may assume that $P=\bf i$. In the upper
half-space model of $\H^3$, we identify $P$ with $(0,0,1)$. Let $\alpha_i=x_i+iy_i,~ y_i>0.$
\noindent  As a point of the boundary of $\H^3$, $\alpha_i$ is represented by
$(x_i,y_i,0)$, while its conjugate $\overline{\alpha_i}$ is represented
by $(x_i,-y_i,0)$.

\

\noindent  Let ${\bf u}(P,\alpha_i)$ and ${\bf u}(P,\overline{\alpha_i})$ denote the unit tangent
vectors at $P$ pointing toward $\alpha_i$ and $\overline{\alpha_i}$,
respectively. We compute ${\bf u}(P,\alpha_i)$ by making the geodesic from $P$
to $\alpha_i$ explicit. Set
\[
r_i=|\alpha_i|=\sqrt{x_i^2+y_i^2}.
\]
\noindent  The geodesic from $P$ to $\alpha_i$ lies in the vertical half-plane
\[
\Pi_i=
\left\{
\left(\rho\frac{x_i}{r_i},
      \rho\frac{y_i}{r_i},t\right):
\rho\in\R,\ t>0
\right\}.
\]
In the coordinates $(\rho,t)$ on $\Pi_i$, the point $P$ is $(0,1)$ and
$\alpha_i$ is the ideal boundary point $\rho=r_i,~t=0$. The geodesic joining
them is the semicircle
\[
\rho(\theta)=c_i+R_i\cos\theta,\qquad
t(\theta)=R_i\sin\theta,
\]
where
\[
c_i=\frac{r_i^2-1}{2r_i},
\qquad
R_i=\frac{r_i^2+1}{2r_i}.
\]
The point $P$ corresponds to the angle $\theta_0$ satisfying
\[
\cos\theta_0=\frac{1-r_i^2}{1+r_i^2},
\qquad
\sin\theta_0=\frac{2r_i}{1+r_i^2}.
\]
Since
\[
ds=-\frac{d\theta}{\sin\theta}
\]
along the geodesic oriented toward the ideal endpoint $r_i$, differentiation
at $P$ gives
\[
\left(\frac{d\rho}{ds},\frac{dt}{ds}\right)\bigg|_P
=
\left(
\frac{2r_i}{1+r_i^2},
\frac{r_i^2-1}{1+r_i^2}
\right).
\]
\noindent  Rotating the horizontal component back in the direction
$(x_i,y_i)/r_i$ gives
\[
{\bf u}(P,\alpha_i)
=
\frac{1}{1+x_i^2+y_i^2}
\left(
2x_i,\,
2y_i,\,
x_i^2+y_i^2-1
\right).
\]
Replacing $y_i$ by $-y_i$ gives
\[
{\bf u}(P,\overline{\alpha_i})
=
\frac{1}{1+x_i^2+y_i^2}
\left(
2x_i,\,
-2y_i,\,
x_i^2+y_i^2-1
\right).
\]
Hence
\[
{\bf u}(P,\alpha_i)+{\bf u}(P,\overline{\alpha_i})
=
\left(
\frac{4x_i}{1+x_i^2+y_i^2},
\,0,\,
\frac{2(x_i^2+y_i^2-1)}
     {1+x_i^2+y_i^2}
\right).
\]

\noindent  We now express the same tangent vector using the hyperboloid model of
$\H^2$. The point $\alpha_i=x_i+iy_i$ is represented by
\[
\mathbf{x}_i
=
\left(
\frac{1-x_i^2-y_i^2}{2y_i},
\frac{x_i}{y_i},
\frac{1+x_i^2+y_i^2}{2y_i}
\right).
\]
If $t_i=d_H(P,\alpha_i)$, then
\[
\cosh t_i
=
\frac{1+x_i^2+y_i^2}{2y_i}.
\]
From the equation (\ref{geodesic})
\[
\mathbf{x}_i
=
\cosh(t_i)P+\sinh(t_i){\bf v}_i,
\]
we obtain
\[
\tanh(t_i){\bf v}_i
=
\left(
\frac{1-x_i^2-y_i^2}
     {1+x_i^2+y_i^2},
\frac{2x_i}
     {1+x_i^2+y_i^2},
0
\right).
\]

\noindent  The last two formulas use different coordinate realizations of the same
tangent plane. At the base point $P$, let
\[
\iota:
T_P\H^2_{\mathrm{hyp}}
\longrightarrow
T_P\H^2_{\mathrm{uhs}}
\]
\noindent  denote the natural identification between the tangent space in the
hyperboloid model and the tangent space of the vertical copy of $\H^2$
in the upper half-space model. In the coordinates used above,
\[
\iota(a,b,0)=(b,0,-a).
\]
Therefore
\[
\begin{aligned}
\iota\bigl(2\tanh(t_i){\bf v}_i\bigr)
&=
\left(
\frac{4x_i}{1+x_i^2+y_i^2},
\,0,\,
\frac{2(x_i^2+y_i^2-1)}
     {1+x_i^2+y_i^2}
\right)\\
&=
{\bf u}(P,\alpha_i)+{\bf u}(P,\overline{\alpha_i}).
\end{aligned}
\]
\noindent Thus, under the natural identification of the two tangent-space models,
\[
{\bf u}(P,\alpha_i)+{\bf u}(P,\overline{\alpha_i})
=
2\tanh(t_i){\bf v}_i.
\]
Summing over the conjugate pairs gives
\[
\sum_{i=1}^n
\bigl(
{\bf u}(P,\alpha_i)+{\bf u}(P,\overline{\alpha_i})
\bigr)
=
2\sum_{i=1}^n\tanh(t_i){\bf v}_i.
\]

\noindent  Assume first that $F$ is stable.  By Corollary \ref{tangent vectors}, $\z_{\J}(F)$ is the unique point of $\H^3$ at which the
unit tangent vectors toward all the roots of $F$ sum to zero. The identity
above shows that, on the invariant copy of $\H^2$, this condition is
equivalent to
\[
\sum_{i=1}^n\tanh(t_i){\bf v}_i=0.
\]
Therefore
\[
P=\z_{\J}(F)
\quad\Longleftrightarrow\quad
\sum_{i=1}^n\tanh(t_i){\bf v}_i=0.
\]

\noindent If $F$ is not stable, then as recorded in the remark of Section~\ref{sec:julia} all upper half-plane roots coincide, say $\alpha_1=\alpha_2=...=\alpha_n=\alpha$. Under the convention $\z_{\J}(F)=\alpha$ adopted there, the stated equilibrium condition is immediate.
\end{proof}

\begin{remark}
\noindent  The proposition gives a useful geometric interpretation of Julia's zero for
real forms with no real roots. Each conjugate pair
$\alpha_i,\overline{\alpha_i}$ contributes a vector tangent to $\H^2$ whose
magnitude is $2\tanh(t_i)$. This is because the components of $ {\bf u}_i(P,\alpha_i), {\bf u}_i(P,\overline{\alpha_i})$ that are orthogonal to $T_{P}\H^2$ cancel out, while the tangential components are both equal to $\tanh t_i$. Thus, after pairing conjugate roots, the
Cremona--Stoll equilibrium condition in $\H^3$ becomes the weighted
equilibrium
\[
\sum_{i=1}^n\tanh(t_i){\bf v}_i=0
\]
in $\H^2$.

\

\noindent  The fact that $\tanh(t)$ is bounded will be important later. In the
resulting $\H^2$-equilibrium, a root moving far from $P$ contributes a
vector whose magnitude approaches $1$, rather than growing without bound.
This is the geometric source of the bounded-influence behavior of Julia's
zero proved in Section~\ref{sec:explicit-asymptotic}.~$\dagger$
\end{remark}

\

\section{The reduction of real forms via the hyperbolic center of mass}\label{sec:hypcenter}

\

\noindent In this section we introduce an alternative zero map for binary forms with real coefficients and no real roots. It is based on the notion of {\bf hyperbolic center of mass} in hyperbolic spaces. We focus on $\H^2$ which is the case of interest for us, but the general case is straightforward. Our treatment follows closely that of \cite{Gal}.

\

\noindent Let ${\mathcal R}_{2n,0}\subset V_{2n,\R}$ denote binary forms of degree $2n$ with real coefficients and no real roots. Every $F(X,Z)\in {\mathcal R}_{2n,0}$ can be factored 
\[F(X,Z)=a_0\prod_{i=1}^n Q_{\alpha_i}(X,Z)\] where 
\[Q_{\alpha_i}(X,Z)=(X-\alpha_iZ)(X-\overline{\alpha_i}Z)\] with ${\alpha}_i\in \H^2$ and $a_0\in \R$. The upper half-plane points $\alpha_1,...\alpha_n$ have a hyperbolic center of mass \[\mathcal C_{F}:={\mathcal C}_H(\alpha_1,\alpha_2,...,\alpha_n),\] where we identify $\H^2$ with the hyperboloid model by the isometry of Section~\ref{sec:hypplane}, so that Definition~\ref{def:centerofmass} and Proposition~\ref{centerfromgeometry} apply verbatim. Transporting Proposition~\ref{centerfromgeometry} through that isometry, $\mathcal C_F$ is the unique point $P\in\H^2$ at which \[\sum_{i=1}^n \sinh(t_i){\bf v}_i=0,\] where $t_i=d_H(P,\alpha_i)$ and ${\bf v}_i$ is the unit tangent vector at $P$ toward $\alpha_i$, with zero contribution when $P=\alpha_i$. 

\
\begin{defn}
The map $\z_H:{\mathcal R}_{2n,0}\rightarrow \H^2,~\z_H(F)={\mathcal C}_{F}$ is called \break \underline{the hyperbolic zero map}. The quadratic form
\[Q_{\mathcal{C}_{F}}(X,Z)=(X-{\mathcal C}_F Z)(X-\overline{{\mathcal C}_F}Z)\] corresponding to $\mathcal C_{F}$
is called \underline{the hyperbolic quadratic polynomial} of $F$.~$\dagger$
\end{defn}

\noindent The reduction theory based on the hyperbolic center of mass proceeds as before. Let $F(X,Z)$ be a real binary form with no real roots. If ${\mathcal C}_{F}\in \F$ then $F$ is reduced. Otherwise, let $M\in \Sl_2(\Z)$ such that ${\mathcal C}_F\cdot M \in \F$. The form $F$ reduces to $F^M(X,Z)$.

\

\section{Comparing the Two Zero Maps}

\

\noindent Let $F(X,Z)$ be a binary form with real coefficients and no real roots. Write $\alpha_1,\ldots,\alpha_n$ for its upper-half-plane roots. Recall

\begin{enumerate}

\item \noindent Julia's zero map of $F$ is the point  $\z_{\J}(F)\in \H^2$ where $$ \sum_{i=1}^n (\tanh t_i) {\bf v}_i=0$$ 

\item The hyperbolic zero map of $F$ is the point $\z_H(F)\in \H^2$ where   $$\sum_{i=1}^n (\sinh t_i) {\bf v}_i=0,$$ by Proposition~\ref{centerfromgeometry} and the identification made in Section~\ref{sec:hypcenter}.

\end{enumerate}

\noindent In both equations $t_i$ is the hyperbolic distance from the zero map value to $\alpha_i$ and ${\bf v}_i$ is the unit vector from the zero map value to the point $\alpha_i$. The two zero maps are expressed as two hyperbolic equilibrium laws with different radial weights.

\

\noindent The coincidence problem reduces to when a point $P$ satisfies simultaneously $$ \sum_{i=1}^n (\tanh t_i){\bf v}_i=\sum_{i=1}^n (\sinh t_i){\bf v}_i=0.$$ 

\begin{prop} (Symmetry Principle) Let $F$ be a stable totally complex binary form with real coefficients. Let $S=\{\alpha_1,...,\alpha_n\}\subset {\H}^2$ be the multiset of upper half-plane roots of $F$; since $S$ determines $F$ up to a nonzero real scalar and both zero maps are unchanged by such scaling, we write $\z_{\J}(S)$ and $\z_H(S)$ for $\z_{\J}(F)$ and $\z_H(F)$, and $S\cdot M$ for the root multiset of $F^M$. Let $G\leq \Sl_2(\R)$ preserve $S$ as a multiset, i.e. $S\cdot M=S$ for every $M\in G$. Assume also that the induced action of $G$ on ${\H}^2$ has a unique common fixed point $P$ meaning $P\cdot M=P$ for every $M\in G$. Then $\z_{\J}(F)=\z_H(F)=P$.

\end{prop}

\begin{proof} Let  $\z$ denote either zero map value of $F$ and let $M\in G$. Since $\z$ is equivariant, then  $$\z(S\cdot M)=\z(S)\cdot M$$ But $S\cdot M=S$ hence $$\z(S)=\z(S)\cdot M$$ It follows that $
\z(S)$ is fixed by every $M\in G$. Since the induced action of $G$ has a unique common fixed point in $\H^2$ we obtain $\z_{\J}(S)=P$ and $\z_H(S)=P$, so $\z_{\J}(F)=\z_H(F)$.

\end{proof}

\begin{cor} Let $F$ and $S$ be as above. If $S$ is invariant under the hyperbolic half-turn about a point $P\in\H^2$, then $\z_{\J}(S)=\z_H(S)=P$.~$\dagger$
\end{cor}

\

\begin{cor}\label{rotation} Let $F$ and $S$ be as above. Suppose there is $M\in \Sl_2(\R)$ such that its induced action in $\H^2$ is an order three rotational symmetry about a point $P\in \H^2$ and $S\cdot M=S$. Then $\z_{\J}(S)=\z_H(S)=P$.~$\dagger$
\end{cor}

\

\noindent These two corollaries will account for the converse statements in the low-degree coincidence results below.

\

\section{Coincidence of Zero Maps in Low Degrees}

\

\begin{prop} Let $F$ be a binary quartic with real coefficients and no real roots. Then $\z_{\J}(F)=\z_H(F)$=hyperbolic midpoint of the two roots of $F$ in $\H^2$, counted with multiplicity.
\end{prop}
\begin{proof}  If the two upper-half-plane roots coincide, $F$ is nonstable and both zero maps equal the common root by the convention of Section~\ref{sec:julia}; that root is its own midpoint. Otherwise let $P$ be the hyperbolic midpoint of the two distinct roots $\alpha_1,\alpha_2 \in \H^2$. Then use the half-turn corollary of the previous section to conclude that $\z_{\J}(F)=\z_H(F)=P$.
\end{proof}

\begin{prop}
Let $F$ be a binary sextic with real coefficients and no real roots. Let $\alpha_1,\alpha_2,\alpha_3$ be its upper-half-plane roots, counted with multiplicity. Then $\z_{\J}(F)=\z_H(F)$ if and only if either the roots form a noncollinear equilateral hyperbolic triangle, or, after relabeling, $\alpha_2$ is the hyperbolic midpoint of $\alpha_1$ and $\alpha_3$. The midpoint case includes three coincident roots.
\end{prop}

\begin{proof} If all three roots coincide, $F$ is nonstable and both zero maps equal that root by the convention of Section~\ref{sec:julia}; the midpoint condition holds trivially. Suppose henceforth that they do not all coincide. Suppose first that $\z_{\J}(F)=\z_H(F)$ and $\alpha_1,\alpha_2,\alpha_3$ lie on the same hyperbolic geodesic. The common center lies on this geodesic: orthogonal projection onto it strictly decreases the distance to every root for a point off the geodesic, and hence decreases the hyperbolic-center functional. Choose an orientation in this geodesic and let $t_i\in \R$ be the signed hyperbolic distance of $\alpha_i$ from the common zero map value $P=\z_{\J}(F)=\z_H(F)$. The two equilibrium laws give $$\sinh t_1+\sinh t_2+\sinh t_3=0$$ and $$\tanh t_1+\tanh t_2+\tanh t_3=0$$ 

\

\noindent Assume that all $t_i\neq 0$. After relabeling and reversing orientation if necessary, assume that $$t_1=a>0,~t_2=b>0, ~t_3=-c<0.$$ We obtain $$\sinh c=\sinh a+\sinh b,~\tanh c=\tanh a+\tanh b$$ But $$\sinh(a+b)=\sinh a \cosh b+\cosh a\sinh b>\sinh a+\sinh b=\sinh c.$$ It follows that $c<a+b$ since $\sinh t$ is strictly increasing. But then $$\tanh c<\tanh (a+b)=\frac{\tanh a+\tanh b}{1+\tanh a\tanh b}<\tanh a+\tanh b=\tanh c.$$ This is a contradiction. It follows that one, say $t_2=0$. Then $\sinh t_1+\sinh t_3=0$, hence, after relabeling, $t_1=-t_3>0$ (the all-coincident case was already handled). Thus $\alpha_2$ is the common zero map and $\alpha_1,\alpha_3$ are on the opposite sides of $\alpha_2$ and equidistant from $\alpha_2$.

\

\noindent Now suppose that $\z_{\J}(F)=\z_H(F)$ and $\alpha_1,\alpha_2,\alpha_3$ are \underline{not} in the same hyperbolic geodesic and let $P$ be the common zero map value. Here put $t_i=d_H(P,\alpha_i)\geq0$. None is zero: if $P=\alpha_i$, the equilibrium law, with its zero-contribution convention, forces the other two directions to be opposite, contradicting noncollinearity. Thus all directions ${\bf v}_i$ are defined. At $P$ we have $$\sum_{i=1}^3(\sinh t_i)v_i=0,~\sum_{i=1}^3(\tanh t_i)v_i=0$$ Because the roots $\alpha_1,\alpha_2,\alpha_3$ are not collinear, the directions ${\bf v}_i$ are not all parallel; they span the two-dimensional tangent space of $\H^2$. Hence, the kernel has dimension $3-2=1$. It follows that $$\left<\sinh t_1,\sinh t_2,\sinh t_3\right>=\lambda \left<\tanh t_1,\tanh t_2,\tanh t_3\right>.$$ Hence, $$\cosh t_1=\cosh t_2=\cosh t_3=\lambda,$$ i.e. $$t_1=t_2=t_3=t$$ It follows that $${\bf v}_1+{\bf v}_2+{\bf v}_3=0,$$ and since they are all unit vectors the angle between each pair among ${\bf v}_1,{\bf v}_2,{\bf v}_3$ is $2\pi/3$. It follows that the three hyperbolic triangles $$\triangle(P\alpha_1\alpha_2),~\triangle(P\alpha_1\alpha_3),~\triangle(P\alpha_3\alpha_2)$$ are congruent by $SAS$. Hence $\triangle(\alpha_1,\alpha_2,\alpha_3)$ is equilateral.

\begin{figure}
\centering
\includegraphics[width=0.95\textwidth]{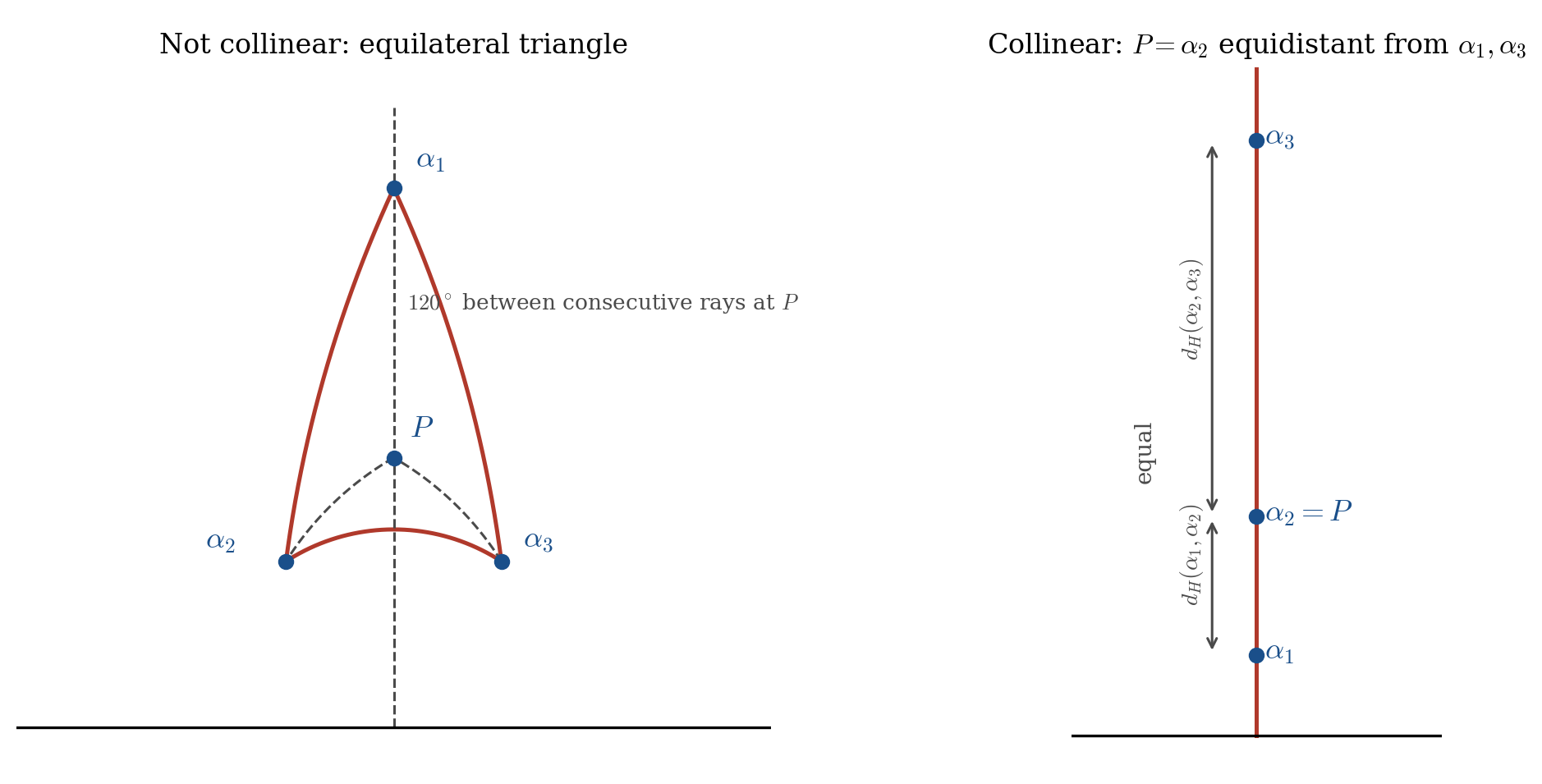}
 \caption{Two cases of sextics where $\z_{\J}(F)=\z_H(F)$ }
   \label{fig2}
\end{figure}

\

\noindent The converse statements follow easily from the corollaries of the previous section.

\end{proof}
\

\begin{prop} Let $F$ be a binary octic with real coefficients and no real roots. Assume that the roots $\alpha_1,\alpha_2,\alpha_3,\alpha_4$ in $\H^2$ lie on the same hyperbolic geodesic. Then $\z_{\J}(F)=\z_H(F)=P$ iff those four points can be split in two pairs each having $P$ as their hyperbolic midpoint.
\end{prop}

\begin{proof} If the zero maps coincide, their common value lies on the root geodesic by the projection argument above. Let $P$ be a point on that geodesic and $t_i$ the signed hyperbolic distance of $\alpha_i$ from $P$ for $i=1,2,3,4$. 

\

\noindent Assume that $P$ is simultaneously the hyperbolic and the Julia zero map value. Recall that the following identities hold \[\sum_{i=1}^4 \sinh t_i=0,~\text{and}~\sum_{i=1}^4 \tanh t_i=0\] If we let $x_i=\tanh t_i\in (-1,1)$ then $\displaystyle{\sinh t_i=\frac{x_i}{\sqrt{1-x_i^2}}}$. Let $\displaystyle{h(x)=\frac{x}{\sqrt{1-x^2}}}$; it is odd in $(-1,1)$ and is a strictly convex function in $(0,1)$. With this notation, $P$ is the common zero map value iff  \[\sum_{i=1}^4 x_i=0,~\text{and}~\sum_{i=1}^4 h(x_i)=0\]

\noindent First assume that all $x_i$ are nonzero. We show that two of $\alpha_1,\alpha_2,\alpha_3,\alpha_4$ are on one side of $P$ and the other two are on the other side of $P$. Indeed, assume not and let $x_1=a,x_2=-b,x_3=-c,x_4=-d$ with all $a,b,c,d$ positive. Then $1>a=b+c+d$ and since $h(0)=0$ and $h(x)$ is strictly convex we obtain $$h(b+c+d)>h(b)+h(c)+h(d)$$ But $0=h(a)+h(-b)+h(-c)+h(-d)=h(a)-h(b)-h(c)-h(d)$ hence $$h(a)=h(b)+h(c)+h(d)$$ and we obtain a contradiction.

\noindent So, suppose there are two on each side: $x_1=a,x_2=b,x_3=-c,x_4=-d$ with all $1>a,b,c,d>0$. Then $$0<a+b=c+d=S<2$$ and $$h(a)+h(b)=h(c)+h(d)$$ Let $f(x)=h(x)+h(S-x)$ with domain $(\text{max}\{0, S-1\},\text{min}\{1,S\}).$ We see that  $f(a)=f(c)$. Now $f'(x)=h'(x)-h'(S-x)$ and $h'$ is increasing in $(0,1)$. Since $S/2<1$ we find that $f'(x)<0$ i.e. $f(x)$ is decreasing on the part of its domain where $x<S/2$ and $f'(x)>0$ i.e. $f(x)$ is increasing on the part where $x>S/2$. We also have that $f(x)=f(S-x)$. Hence, either $c=a$ or $c=S-a=b$. It follows that $$\{a,b\}=\{c,d\}$$ as multisets. 

\noindent But $x=\tanh t$ is odd and injective, hence $$\{t_1,t_2,t_3,t_4\}=\{r,-r,s,-s\}$$ for some $r,s\geq 0$. After relabeling $d_H(P,\alpha_1)=d_H(P,\alpha_2),~d_H(P,\alpha_3)=d_H(P,\alpha_4)$ with the members of each pair lying on opposite sides of $P$, unless both equal $P$.

\begin{figure}
\centering

\includegraphics[width=0.95\textwidth]{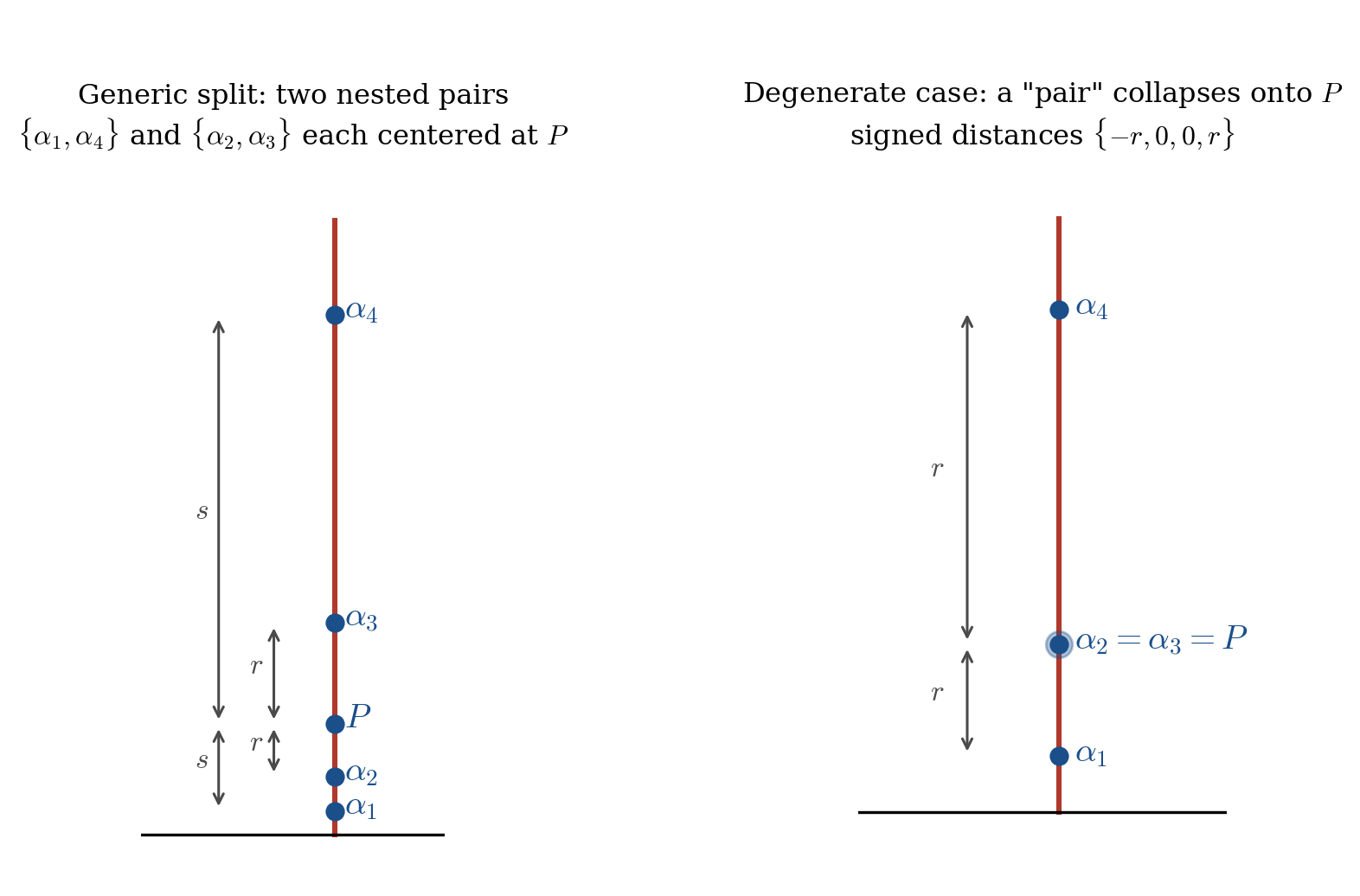}

 \caption{Two cases of octics with $\z_{\J}(F)=\z_H(F) $ }
   \label{fig3}
\end{figure}

\

\noindent So far we assumed that all $x_i$ are nonzero. If one $x_i=0$,  the remaining three satisfy $\sum x_i=0$ and $\sum h(x_i)=0$. The three-point collinear argument in the previous proposition implies that those remaining three are of the form $\{-r,0,r\}$ in signed $t$-coordinates. Hence the four are $\{-r,0,0,r\}$ which still split into two midpoint pairs. Here $r\geq0$, so this also covers the case in which all four roots coincide; in that case $F$ is nonstable and both zero maps equal the common root by the convention of Section~\ref{sec:julia}. Alternatively, with at least two zero coordinates, the equation $\sum x_i=0$ makes the remaining two coordinates opposite, giving the same pairing.

\

\noindent The converse statement follows from the half-turn corollary of the previous section. 

\end{proof}

\

\section{Explicit Formula and Asymptotic Comparison}
\label{sec:explicit-asymptotic}

\

\noindent The preceding sections compare the Julia and hyperbolic zero maps through
their equilibrium laws. For totally complex real forms, the hyperbolic zero
has an additional advantage: it can be computed explicitly from the real
quadratic factors of the form. We recall this formula first, and then use it
to compare the two reductions computationally and in degenerating families.

\begin{prop}[{\cite{ES}}]
\label{prop:explicit-hyperbolic-zero}
Let
\[
F(X,Z)=\prod_{i=1}^r\left(X^2+a_iXZ+b_iZ^2\right)
\]
be a totally complex real binary form, and put
\[
d_i=\sqrt{4b_i-a_i^2},\qquad i=1,\ldots,r.
\]
If
\[
\z_H(F)=t+iu,
\]
then
\[
t=-\frac12
\frac{\displaystyle\sum_{i=1}^r a_i/d_i}
     {\displaystyle\sum_{i=1}^r 1/d_i},
\qquad
u^2=
\frac{\displaystyle\sum_{i=1}^r b_i/d_i}
     {\displaystyle\sum_{i=1}^r 1/d_i}
-t^2.
\]
\noindent  In particular, the hyperbolic center quadratic is defined over
\[
L(d_1,\ldots,d_r),
\qquad
L=\Q(a_1,\ldots,a_r,b_1,\ldots,b_r).~\dagger
\]
\end{prop}

\

\begin{remark} For forms given as products of real quadratic factors, these formulas give
the hyperbolic zero directly from the coefficients of the factors. In
contrast, Julia's zero is in general determined by a nonlinear minimization
or equilibrium problem. Thus, once the real quadratic factorization is
known, the hyperbolic zero avoids the nonlinear computation required for
Julia's zero.~$\dagger$
\end{remark}

\noindent For a specified binary form $F(X,Z)=\sum_{j=0}^N a_jX^{N-j}Z^j$, we use the coefficient height
\[
H(F)=\max_{0\leq j\leq N}|a_j|.
\]

\begin{exa}
Consider the sextic
\[
F(X,Z)
=
(X^2+21XZ+111Z^2)
(X^2+10XZ+169Z^2)
(X^2+2XZ+128Z^2).
\]
Numerically,
\[
\z_H(F)\approx -9.51918+5.01648{\bf i},
\qquad
\z_{\J}(F)\approx -7.40710+8.49127\bf i.
\]
\noindent  The two points lie in different translates of the standard fundamental
domain. Hyperbolic reduction translates by $+10$, giving
\[
F_{\H}(X,Z)=F(X-10Z,Z),
\]
\noindent  whereas Julia reduction translates by $+7$, giving
\[
F_{\J}(X,Z)=F(X-7Z,Z).
\]
Explicitly,
\[
\begin{aligned}
F_{\H}(X,Z)
={}&X^6-27X^5Z+530X^4Z^2-4593X^3Z^3\\
&\quad+30587X^2Z^4+30030XZ^5+35152Z^6,
\end{aligned}
\]
while
\[
\begin{aligned}
F_{\J}(X,Z)
={}&X^6-9X^5Z+260X^4Z^2-123X^3Z^3\\
&\quad+11795X^2Z^4+137304XZ^5+313612Z^6.
\end{aligned}
\]
\noindent  Thus in this example the hyperbolically reduced representative has height
$35152$, whereas the Julia-reduced representative has height $313612$,
approximately $8.9$ times larger.
\end{exa}

\noindent  The next result gives a different kind of comparison. It shows that a
strict majority of roots confined to a compact set controls Julia's zero,
regardless of the positions of the remaining roots.

\begin{prop}
\label{prop:majority}
Let $n\geq3$ and let $K\subset\H^2$ be compact. There exists a compact set
\[
K'=K'(K,n)\subset\H^2
\]
\noindent  with the following property. If $F(X,Z)$ is a stable, totally complex real
binary form of degree $2n$ and more than half of its upper-half-plane roots,
counted with multiplicity, lie in $K$, then
\[
\z_{\J}(F)\in K'.
\]
\noindent  Equivalently, it is enough that at least
\[
m_0=\left\lfloor\frac n2\right\rfloor+1
\]
of the upper-half-plane roots lie in $K$.
\end{prop}

\begin{proof}
Suppose otherwise. Then there is a sequence of stable, totally complex real
binary forms $F_r$ of degree $2n$ such that, after relabeling,
\[
\alpha_{1,r},\ldots,\alpha_{m_0,r}\in K,
\qquad
m_0=\left\lfloor\frac n2\right\rfloor+1,
\]
while
\[
P_r:=\z_{\J}(F_r)
\]
escapes every compact subset of $\H^2$, for example $d_H(P_r,{\bf i})>r$ for all $r$.

\

\noindent For $1\leq j\leq m_0$, put
\[
t_{j,r}=d_H(P_r,\alpha_{j,r}),
\]
\noindent and let ${\bf v}_{j,r}$ be the unit tangent vector at $P_r$ pointing toward
$\alpha_{j,r}$. Since $K$ is compact, let $R_K=\text{max}_{z\in K}d_H({\bf i},z)$. From $d_H(P_r,{\bf i})>r$ follows $$t_{j,r}\geq d_H(P_r,{\bf i})-d_H({\bf i}, \alpha_{j,r})\geq r-R_K$$ for all large $r$. Thus
\[
t_{j,r}\longrightarrow\infty~~\text{hence}~~\tanh(t_{j,r})\longrightarrow 1.
\]
uniformly for $1\leq j\leq m_0$. 

\

\noindent The directions ${\bf v}_{j,r}$ also become asymptotically parallel. Indeed, let
$\theta_{jk,r}$ be the angle at $P_r$ between the geodesics toward
$\alpha_{j,r}$ and $\alpha_{k,r}$. The hyperbolic law of cosines gives
\[
\cos\theta_{jk,r}
=
\frac{\cosh t_{j,r}\cosh t_{k,r}
-\cosh d_H(\alpha_{j,r},\alpha_{k,r})}
{\sinh t_{j,r}\sinh t_{k,r}}.
\]
The distances $d_H(\alpha_{j,r},\alpha_{k,r})$ are uniformly bounded,
whereas $t_{j,r},t_{k,r}\to\infty$, so
\[
\theta_{jk,r}\longrightarrow0.
\]
Taking ${\bf v}_r={\bf v}_{1,r}$, we therefore have
\[
\|{\bf v}_{j,r}-{\bf v}_r\|\longrightarrow0,
\qquad 1\leq j\leq m_0.
\]
It follows that
\[
\left\|
\sum_{j=1}^{m_0}\tanh(t_{j,r}){\bf v}_{j,r}
\right\|
\longrightarrow m_0.
\]

\noindent  On the other hand, Julia's equilibrium condition gives
\[
\sum_{j=1}^{m_0}\tanh(t_{j,r}){\bf v}_{j,r}
=
-\sum_{j=m_0+1}^{n}\tanh(t_{j,r}){\bf v}_{j,r}.
\]
Since every vector on the right has norm at most $1$,
\[
\left\|
\sum_{j=m_0+1}^{n}\tanh(t_{j,r}){\bf v}_{j,r}
\right\|
\leq n-m_0.
\]
Passing to the limit gives
\[
m_0\leq n-m_0,
\]
contradicting $m_0>n/2$. Hence $\z_{\J}(F)$ remains in a compact set
depending only on $K$ and $n$.
\end{proof}

\begin{remark}
\label{rem:majority-sharp}
The strict-majority hypothesis is sharp. Fix $\alpha\in\H^2$ and let
$\beta_M\in\H^2$ lie on a fixed geodesic through $\alpha$, with
\[
d_H(\alpha,\beta_M)\longrightarrow\infty.
\]
\noindent  Consider a configuration consisting of $m$ copies of $\alpha$ and $n-m$
copies of $\beta_M$, where $0<m<n$. The corresponding totally complex real
form is stable, since each complex root has multiplicity strictly less than
$n$.

\noindent  The defining minimization involves only distances, hence is invariant
under every isometry, including the reflection fixing the geodesic through
$\alpha$ and $\beta_M$; since the minimizer is unique it is fixed by that
reflection, so $P_M$ lies on the geodesic, and the equilibrium condition
places it between $\alpha$ and $\beta_M$. Put
\[
a_M=d_H(P_M,\alpha),
\qquad
b_M=d_H(P_M,\beta_M).
\]
Then
\[
a_M+b_M=d_H(\alpha,\beta_M)\longrightarrow\infty,
\]
and Julia's equilibrium condition reduces to
\[
m\tanh(a_M)=(n-m)\tanh(b_M).
\]

\noindent If $m=n/2$, then $a_M=b_M$, so
\[
a_M=b_M=\frac12d_H(\alpha,\beta_M)\longrightarrow\infty.
\]
\noindent  Thus $P_M$ is the hyperbolic midpoint of $\alpha$ and $\beta_M$ and
escapes every compact subset of $\H^2$.

\

\noindent If $m<n/2$ and $a_M$ remained bounded along a subsequence, then
$b_M\to\infty$, and hence
\[
(n-m)\tanh(b_M)\longrightarrow n-m>m,
\]
whereas
\[
m\tanh(a_M)\leq m,
\]
\noindent  contradicting the equilibrium equation. Thus $a_M\to\infty$ in this case
as well. Consequently, the conclusion of Proposition~\ref{prop:majority}
can fail whenever $m\leq n/2$, and the strict-majority condition is
optimal.
\end{remark}

\noindent  The contrast with the hyperbolic zero is already visible when a single
root escapes, and persists when several escape simultaneously.

\begin{prop}
\label{prop:escaping-minority}
Let $n\geq3$, let $K\subset\H^2$ be compact, and let
\[
m>\frac n2,\qquad k=n-m\geq1.
\]
For each $M>0$, let $F_M(X,Z)$ be a stable, totally complex real binary
form of degree $2n$ whose upper-half-plane roots satisfy
\[
\alpha_i^{(M)}=x_i^{(M)}+iy_i^{(M)}\in K,
\qquad i=1,\ldots,m,
\]
and
\[
\alpha_{m+j}^{(M)}=-M+iy_{j,M},
\qquad j=1,\ldots,k,
\]
where
\[
y_{j,M}\in[c,C]
\]
for fixed constants $0<c\leq C<\infty$. Then:

\begin{enumerate}
\item There exists a compact set $K'=K'(K,n)\subset\H^2$ such that $\z_{\J}(F_M)\in K'$
for all $M>0$.

\item Let $\z_H(F_M)=t_M+iu_M.$ Then there are constants $0<c_1<c_2<\infty$ and $c_3>0$, depending only
on $K,n,m,c,C$, such that, for all sufficiently large $M$,
\[
c_1M\leq |t_M|\leq c_2M,
\qquad
u_M\geq c_3M.
\]
In particular,
\[
\frac{\operatorname{Im}\z_{\J}(F_M)}
     {\operatorname{Im}\z_H(F_M)}
\longrightarrow0.
\]
\end{enumerate}
\end{prop}

\begin{proof}
Part (1) follows immediately from Proposition~\ref{prop:majority}.

\

\noindent  For part (2), denote
\[
A_M:=\sum_{i=1}^{m}\frac{x_i^{(M)}}{y_i^{(M)}},
\qquad
B_M:=\sum_{i=1}^{m}\frac1{y_i^{(M)}},
\qquad
D_M:=\sum_{j=1}^{k}\frac1{y_{j,M}}.
\]
Since $K$ is compact and $y_{j,M}\in[c,C]$, there are positive constants
$A_0, B_0, B_1, D_0,D_1$ independent of $M$ such that
\[
|A_M|\leq A_0,
\qquad
0<B_0\leq B_M\leq B_1,
\qquad
0<D_0\leq D_M\leq D_1.
\]
Proposition~\ref{prop:explicit-hyperbolic-zero} gives
\[
t_M=\frac{A_M-MD_M}{B_M+D_M}.
\]
Thus
\[
t_M=-\lambda_M M+O(1),
\qquad
\lambda_M=\frac{D_M}{B_M+D_M}.
\]
The bounds above imply that there are constants
\[
0<\lambda_-\leq\lambda_M\leq\lambda_+<1
\]
independent of $M$. Hence
\[
|t_M|\asymp M,
\]
which gives the first pair of inequalities.

\noindent  It remains to estimate $u_M$. Let
\[
\alpha_i^{(M)}=x_i^{(M)}+iy_i^{(M)},
\qquad
w_i=\frac1{y_i^{(M)}},
\qquad
q_i=(t_M-x_i^{(M)})^2+(y_i^{(M)})^2.
\]
The formula for the hyperbolic zero can be rewritten as
\[
u_M^2=
\frac{\displaystyle\sum_{i=1}^n q_iw_i}
     {\displaystyle\sum_{i=1}^n w_i}.
\]
Choose any escaping root, say $\alpha_{m+1}^{(M)}=-M+iy_{1,M}$. Then
\[
q_{m+1}\geq(t_M+M)^2,
\qquad
w_{m+1}\geq\frac1C.
\]
Moreover,
\[
\sum_{i=1}^n w_i=B_M+D_M\leq B_1+D_1.
\]
Since
\[
t_M+M=(1-\lambda_M)M+O(1)
\]
and $1-\lambda_M\geq1-\lambda_+>0$, there is a constant $\delta>0$
such that, for all sufficiently large $M$,
\[
t_M+M\geq\delta M.
\]
Therefore
\[
u_M^2
\geq
\frac{q_{m+1}w_{m+1}}{B_M+D_M}
\geq
\frac{\delta^2}{C(B_1+D_1)}\,M^2.
\]
\noindent  Thus $u_M\geq c_3M$ for some $c_3>0$. The same weighted-average formula gives $u_M=O(M)$, since all root heights are uniformly bounded and all horizontal coordinates are $O(M)$. Hence $u_M\asymp M$, and the hyperbolic distance formula yields
\[
\cosh d_H({\bf i},t_M+{\bf i}u_M)
=\frac{t_M^2+u_M^2+1}{2u_M}\asymp M,
\qquad d_H({\bf i},\z_H(F_M))=\log M+O(1).
\]

\

\noindent  Finally, part (1) implies that $\operatorname{Im}\z_{\J}(F_M)$ is
uniformly bounded above, whereas $u_M\to\infty$ linearly. Hence
\[
\frac{\operatorname{Im}\z_{\J}(F_M)}
     {\operatorname{Im}\z_H(F_M)}
\longrightarrow0.
\]
\end{proof}

\noindent  The following example illustrates the proposition with two escaping
roots, and computes the limit of Julia's zero explicitly.

\begin{exa}
\label{ex:two-escaping-roots}
Let
\[
\alpha_1=2{\bf i},\qquad
\alpha_2=1+2{\bf i},\qquad
\alpha_3=2+2{\bf i},
\]
and
\[
\alpha_4^{(M)}=-M+{\bf i},\qquad
\alpha_5^{(M)}=-M+2{\bf i}.
\]
Consider
\begin{multline*}
F_M(X,Z)
=
(X^2+4Z^2)(X^2-2XZ+5Z^2)(X^2-4XZ+8Z^2)\\
\times
\bigl(X^2+2MXZ+(M^2+1)Z^2\bigr)
\bigl(X^2+2MXZ+(M^2+4)Z^2\bigr).
\end{multline*}
\noindent  Three of the five upper-half-plane roots remain fixed, while the other two
move horizontally to the left at fixed heights $1$ and $2$. For every $M>0$, the $5$ points $\alpha_1,\alpha_2,\alpha_3,\alpha_4(M),\alpha_5(M)$ are pairwise distinct, so $F_M$ is a stable, totally complex, real binary form of degree $10$. Hence Proposition~\ref{prop:escaping-minority} applies with an escaping-minority configuration with $n=5$, $m=3$, and $k=2$.

\

\noindent Let
\[
\z_H(F_M)=t_M+iu_M.
\]
By Proposition~\ref{prop:explicit-hyperbolic-zero},
\[
t_M
=
\frac{\left(\dfrac02+\dfrac12+\dfrac22\right)
      -M\left(\dfrac11+\dfrac12\right)}
     {\left(\dfrac12+\dfrac12+\dfrac12\right)
      +\left(\dfrac11+\dfrac12\right)}
=
\frac{1-M}{2}.
\]
Similarly,
\[
u_M^2
=
\frac{\dfrac42+\dfrac52+\dfrac82
      +(M^2+1)+\dfrac{M^2+4}{2}}{3}
-t_M^2
=
\frac{M^2}{4}+\frac{M}{2}+\frac{43}{12}.
\]
Hence
\[
t_M=-\frac M2+\frac12,
\qquad
u_M\sim\frac M2,
\]
and therefore
\[
|\z_H(F_M)|\asymp M.
\]

\

\noindent  The Julia zero behaves quite differently. We will show that the one-parameter family  $\z_{\J}(F_M)$ has a limit as $M\rightarrow \infty$. 

\

\noindent Indeed, by Proposition~\ref{prop:escaping-minority}(1), $\z_{\J}(F_M)$ remains in a
fixed compact subset of $\H^2$. It is the minimizer of
\begin{multline*}
H_M(u,v)
=
\log(u^2+v^2+4)
+\log\bigl((u-1)^2+v^2+4\bigr)
+\log\bigl((u-2)^2+v^2+4\bigr)\\
+\log\bigl((u+M)^2+v^2+1\bigr)
+\log\bigl((u+M)^2+v^2+4\bigr)
-5\log v.
\end{multline*}
Subtracting the constant $4\log M$ does not change the minimizer, and
uniformly on compact subsets of $\H^2$,
\[
H_M(u,v)-4\log M\longrightarrow H_\infty(u,v),
\]
where
\[
H_\infty(u,v)
=
\log(u^2+v^2+4)
+\log\bigl((u-1)^2+v^2+4\bigr)
+\log\bigl((u-2)^2+v^2+4\bigr)
-5\log v.
\]

\

\noindent Write $\widetilde H_M=H_M-4\log M$ and let $K'$ be a fixed compact subset of $\H^2$ containing the Julia zeros. Take any sequence $M_r\to\infty$ along which $P_r=\z_{\J}(F_{M_r})\to P\in K'$. For every fixed $Q\in\H^2$, global minimality gives
\[
\widetilde H_{M_r}(P_r)\leq\widetilde H_{M_r}(Q).
\]
Local uniform convergence, applied on $K'\cup\{Q\}$, and continuity of $H_\infty$ allow passage to the limit:
\[
H_\infty(P)\leq H_\infty(Q).
\]
Thus $P$ is a global minimizer of $H_\infty$ on $\H^2$, and in particular a critical point. We now show that this critical point is unique and compute it.
\

\noindent Differentiating with respect to
$u$ gives
\[
\frac{u}{u^2+v^2+4}
+\frac{u-1}{(u-1)^2+v^2+4}
+\frac{u-2}{(u-2)^2+v^2+4}=0.
\]
Let $x=u-1$ and $W=v^2+4$. After clearing denominators, the numerator is
\[
x\Bigl(3x^4+(6W-4)x^2+3W^2+1\Bigr).
\]
\noindent  Since $W>4$, the factor in parentheses is strictly positive. Thus every
critical point satisfies $u=1.$ At $u=1$, the critical-point equation in $v$ becomes
\[
2v\left(\frac{2}{v^2+5}+\frac{1}{v^2+4}\right)=\frac5v.
\]
Writing $w=v^2$ and clearing denominators gives
\[
w^2-19w-100=0.
\]
\noindent  This equation has exactly one positive solution,
\[
w=\frac{19+\sqrt{761}}2.
\]
\noindent  Hence $H_\infty$ has a unique critical point in $\H^2$, namely
\[
1+i\sqrt{\frac{19+\sqrt{761}}2}.
\]
\noindent  Every minimizer is a critical point, so the compactness established above
implies that the entire family converges:
\[
\lim_{M\to \infty}\z_{\J}(F_M)=
1+i\sqrt{\frac{19+\sqrt{761}}2}.
\]

\

\noindent  Consequently,
\[
\lim_{M\to \infty} \z_{\J}(F_M)
=
1+i\sqrt{\frac{19+\sqrt{761}}2},
\qquad
|\z_H(F_M)|\asymp M.
\]
\noindent  Thus a strict majority of fixed roots keeps Julia's zero in the interior
even when two roots escape simultaneously, while the hyperbolic zero
escapes at linear scale in the Euclidean upper-half-plane coordinates.
\

\begin{remark} It is worth noting that this limit shares its real part with the Julia zero
of the surviving sextic $G(X,Z)=(X^2+4Z^2)(X^2-2XZ+5Z^2)(X^2-4XZ+8Z^2)$
obtained by discarding the two escaping roots, but not its imaginary part.
Indeed $\z_{\J}(G)$ minimizes $\sum_{i=1}^3\log((u-x_i)^2+v^2+y_i^2)-3\log v$,
whose $u$-equation is identical to that of $H_\infty$, so both minimizers satisfy $u=1$ by the
factorization above. The two functionals differ only in the coefficient
of $-\log v$ ($H_\infty$ retains $-5\log v$, inherited from all five conjugate
pairs of $F_M$, while $\z_{\J}(G)$'s functional has $-3\log v$), and it
is exactly this extra $-2\log v$ that shifts the limiting height: 
$$\z_{\J}(G)=1+{\bf i}\sqrt{\frac{1+\sqrt{721}}{6}},~~~~~\text{and}~~~~~~\lim_{M\to \infty}\z_{\J}(F_M)=1+{\bf i}\sqrt{\frac{19+\sqrt{761}}{2}}.$$

\noindent The two escaping roots thus
vanish from the real part of the limit entirely, while leaving a permanent
mark on its height.

\end{remark}
\end{exa}

\

\section{Concluding Remarks}

\

\noindent We have compared two reductions of totally complex real binary forms:
Julia's zero and the hyperbolic zero. Although the two constructions come
from different definitions, they can be described by the closely related
equilibrium conditions
\[
\sum_i \tanh(t_i){\bf v}_i=0
\qquad\text{and}\qquad
\sum_i \sinh(t_i){\bf v}_i=0.
\]

\noindent This simple difference between the two equations explains much of what we
have seen in this paper. It also gives a convenient way of deciding when
the two zero maps coincide without having to compute either one explicitly.

\

\noindent The difference between $\tanh t$ and $\sinh t$ becomes particularly clear
when some of the roots move far away. Since $\tanh t$ is bounded, a strict
majority of the roots contained in a fixed compact set is enough to keep
Julia's zero in a compact set, regardless of what happens to the remaining
roots. The strict-majority condition cannot be improved: if only half of
the roots remain controlled, the Julia zero may escape. On the other hand,
an escaping minority can already force the hyperbolic zero to escape at
linear scale in the Euclidean upper-half-plane coordinates. In this sense the threshold $1/2$ gives a sharp distinction
between the influence of distant roots on the two constructions.

\

\noindent There is another difference which is important in computations. The
hyperbolic zero has an explicit formula in terms of the real quadratic
factors of the form. Julia's zero, in general, has to be found by solving
a nonlinear equilibrium problem. The examples considered here show that
the two reductions can produce different representatives of the same
$\Sl_2(\mathbb Z)$-orbit, and in some examples the representative obtained
from the hyperbolic zero has substantially smaller coefficient height.
We do not expect one reduction to always give smaller coefficients than
the other, but it would be useful to understand better when each of them
does.

\

\noindent There are some natural questions left by this comparison. Perhaps the most
immediate one is to characterize, in arbitrary degree, the configurations
for which
\[
\z_{\J}(F)=\z_H(F).
\]

\noindent The low-degree results in this paper suggest that there should be a
geometric answer to this question. It would also be interesting to obtain
effective estimates for the distance between the two zero maps in terms of
the configuration of the roots, and to understand more systematically the
relation between either reduction and coefficient height.

\

\noindent Finally, everything here was done for totally complex real forms. This is
the setting in which the upper-half-plane roots give the two equilibrium
problems considered in the paper. It would be interesting to know whether
a similar comparison can be made for real forms having real roots. For
that, one would first need a suitable extension of the hyperbolic zero to
such configurations.

\

\noindent\textbf{Funding.} No funding was received for conducting this study.


\end{document}